%% file: main.tex
\documentclass[12pt]{amsart}
\usepackage[top=2cm, bottom=2cm, left=2cm, right=2cm]{geometry}
\usepackage{graphicx} 
\usepackage{ascmac}
\usepackage{amsmath,amssymb}
\usepackage{amsthm}
\usepackage{mathtools}

\title{Defining ideals of some numerical semigroup rings with  arithmetic pseudo-Frobenius numbers}
\author{KOU TAKAHASHI}
\input{Thm}
\begin{document}

\begin{abstract}

\input{Abstract}
\end{abstract}

\maketitle

\section{Introduction}

\input{1.Introduction}

\section{Preliminaries}

\input{2.Preliminaries}
\section{Proof of the Main theorem}

\input{3.ProofA}
\section{Further Implications and Examples}

\input{4.proofB}

\input{5.Example}

\input{References}
\Addresses
\end{document}

%% file: Thm.tex
\theoremstyle{plain}
\newtheorem{thm}{Theorem}[section]
\newtheorem{lem}[thm]{Lemma}

\newtheorem{cor}[thm]{Corollary}
\newtheorem{conj}[thm]{Conjecture}
\newtheorem{exam}[thm]{Example}

\theoremstyle{definition}

\newtheorem{dfn}[thm]{Definition}

\theoremstyle{remark}
\newtheorem{rem}[thm]{Remark}

\theoremstyle{plain}
\newtheorem*{thm*}{Theorem}
\newtheorem*{lem*}{Lemma}
\newtheorem*{prop*}{Proposition}
\newtheorem*{cor*}{Corollary}
\newtheorem*{conj*}{Conjecture}
\newtheorem*{clm*}{Claim}

\theoremstyle{definition}
\newtheorem*{ass*}{Assumption}
\newtheorem*{dfn*}{Definition}

\theoremstyle{remark}
\newtheorem*{rem*}{Remark}

\DeclareMathOperator{\PF}{PF}
\DeclareMathOperator{\Ap}{Ap}
\DeclareMathOperator{\PFa}{PF_+}

\numberwithin{equation}{section}

\newcommand{\Addresses}{{
  \bigskip
  \footnotesize

  \textsc{Department of Mathematics,
Faculty of Education,
Waseda University,
1-6-1 Nishi-Waseda, Shinjuku, Tokyo 169-8050, Japan}\par\nopagebreak
  \textit{E-mail address}: \texttt{k\_takahash@akane.waseda.jp}

}}

%% file: Abstract.tex
In this paper, we study defining ideals of numerical semigroup rings.
Let $H$ be a numerical semigroup with multiplicity $a_0$ and embedding dimension $n$.
Assuming $a_0/2+1\leq n$, we prove that the defining ideal of $H$ is determinantal
when the set of pseudo-Frobenius numbers forms an arithmetic sequence of length $n-1$.
This partly resolves a conjecture of Cuong, Kien, Truong and, Matsuoka.

%% file: 1.Introduction.tex
Numerical semigroups and their semigroup rings provide an important class of rings in commutative algebra.
A fundamental problem on numerical semigroup rings is to compute generators of their defining ideals.
As an example,
Herzog \cite{Herzog1970} computed the defining ideal of a numerical semigroup ring with embedding dimension three.
However, when the embedding dimension is four or larger,
the structure of the defining ideal is not fully understood.
Recently the following interesting conjecture was posed by Cuong, Kien, Truong and, Matsuoka.

To introduce the conjecture, let us fix some notation.
Let $k$ be a field and 
\begin{equation*}
    H=\langle a_0,a_1,\ldots,a_{n-1}\rangle=\{\lambda_0a_0+\lambda_1a_1+\cdots+\lambda_{n-1}a_{n-1}:\lambda_0,\lambda_1,\ldots,\lambda_{n-1}\in\mathbb{Z}_{\geq0}\}
\end{equation*}
the numerical semigroup minimally generated by $a_0,a_1,\ldots,a_{n-1}$ with $a_0=\min(H\setminus\{0\})$.
The numerical semigroup ring of $H$, denoted by $k[H]$, is defined as 
\begin{equation*}
    k[H]:=k[t^{a_0},t^{a_1},\ldots,t^{a_{n-1}}]\subseteq k[t].
\end{equation*}
We put $S=k[X_0,X_1,\ldots,X_{n-1}]$ the polynomial ring over $k$ with grading $\deg X_i=a_i$,
and $\varphi:S\to k[H]$ the graded ring homomorphism defined by $\varphi(X_i)=t^{a_i}$ for each $0\leq i\leq n-1$.
The kernel of the homomorphism $\varphi$ is denoted by $I_H$.
For a matrix $M$ whose entries are in $S$,
$\mathrm{I}_2(M)$ denotes the ideal of $S$ generated by 2-minors of $M$.
Let $\PF(H)$ be the set of all pseudo-Frobenius numbers of $H$.
(See Section 2 for the definition of pseudo-Frobenius numbers.)
\begin{conj}\label{conj}
    $($Cuong-Kien-Truong-Matsuoka \cite[Conjecture 1.1]{Kien2025}$)$
    With the notation above, the following conditions are equivalent.
    \def\labelenumi{(\theenumi)}
    \begin{enumerate}
        \item $I_H=\mathrm{I}_2$
        $\begin{psmallmatrix}
            X_0^{l_0}&X_1^{l_1}&\cdots&X_{n-2}^{l_{n-2}}&X_{n-1}^{l_{n-1}}\\
            X_1^{m_1}&X_2^{m_2}&\cdots&X_{n-1}^{m_{n-1}}&X_0^{m_0}
        \end{psmallmatrix}$
    for some integers $l_0,l_1,\ldots,l_{n-1},m_0,m_1,\ldots,m_{n-1}>0$,
    after suitable permutations on $a_0,a_1,\ldots,a_{n-1}$.
    \item The set $\PF(H)$ forms an arithmetic sequence of length $n-1$.
    \end{enumerate}
\end{conj}
The implication from (1) to (2) has already been established in \cite{Kien2020}.
On the other hand, the implication from (2) to (1) remains open in general, but has been verified in certain special cases,
including almost symmetric semigroups  \cite{Goto2018}, semigroups with maximal embedding dimension \cite{Kien2020},
generalized repunit numerical semigroups \cite{Rosales2015}
and, stretched numerical semigroup rings \cite{Kien2025}.

In this paper, we prove that this conjecture holds
when the embedding dimension $n$ of $H$ is relatively large comparing its multiplicity $a_0$.
\begin{thm}\label{thm:1.mainthm}
    Conjecture \ref{conj} holds when $a_0/2+1\leq n$.
\end{thm}
We actually prove a slightly more general statement (see Theorem \ref{thm:3.claimA}). 

The organization of this paper is as follows.
In Section 2, we review basic definitions and properties of
numerical semigroups and their invariants.
Section 3 consists of the proof of our main result.
In Section 4, we discuss further implication which supports Conjecture \ref{conj},
and give a few examples illustrating our main theorem.

%% file: 2.Preliminaries.tex
We introduce notations used in the paper.
Throughout this section, let $H=\langle a_0,a_1,\ldots,a_{n-1}\rangle$ be a numerical semigroup minimally generated by $a_0,a_1,\ldots,a_{n-1}$ with $n\geq3$.
The number $n$ is called the \textbf{embedding dimension} of $H$.
We always assume $a_0=\min (H\setminus \{0\})$, which is called the \textbf{multiplicity} of $H$. 
The kernel $I_H$ of homomorphism $\varphi:S\to k[H]$ is called the \textbf{defining ideal} of $H$. 

We now investigate numerical semigroups by means of Ap\'ery sets and pseudo-Frobenius numbers. 
To begin with, we recall their definitions.
The set 
\begin{equation*}
    \Ap(H)=\{a\in H:a-a_0\not\in H\}
\end{equation*}
is called the \textbf{Ap\'ery set} of $H$.
An integer $p\in \mathbb{Z}_{\geq0}\setminus H$ is called a \textbf{pseudo-Frobenius number} of $H$, if $p+h\in H$ for all $0 < h\in H$.
The set of all pseudo-Frobenius numbers of $H$ is denoted by $\PF(H)$.
The set $\PF(H)+a_0=\{p+a_0:p\in \PF(H)\}$ is denoted by $\PFa(H)$.
Note that $\PFa(H)\subseteq\Ap(H)$.
Let $a, b\in\mathbb{Z}$.
We define the partial order $\leq_H$ on $\mathbb{Z}$ by
$a\leq_Hb$ if $b-a\in H$.

Next, we present some known facts concerning Apéry sets and pseudo-Frobenius numbers.
\begin{lem}\label{lem:2.zinH}
    $($\cite[Proposition 2.19]{Rosales2015}$)$
    For an integer $z\in\mathbb{Z}, z\not\in H$ if and only if there exists $\alpha\in\PF(H)$ such that $\alpha-z\in H$.
\end{lem}
\begin{lem}\label{lem:2.PFexpression}
    $($\cite[Proposition 8]{Assi2020}$)$
    The set $\PFa(H)$ is the set of maximal elements of $\Ap(H)$ with respect to $\leq_H$, that is, 
    \begin{equation*}
        \PFa(H)=\{a\in\Ap(H):b-a\not\in H ~for~all~b\in\Ap(H)\setminus\{a\} \}.
    \end{equation*}
\end{lem}

\begin{cor}\label{cor:2.+eta}
    For all $a\in\Ap(H)$, there exists $\eta\in\Ap(H)$ such that $a+\eta\in\PFa(H)$.
\end{cor}

\begin{lem}\label{lem:2.alphanotinH}
    $($\cite[Lemma 2.5]{Kien2020}$)$
    Suppose $n\geq3$. 
    If there exists $h\geq0$ and $\alpha\in\mathbb{Z}\setminus\{0\}$ such that $\PF(H)=\{h+\alpha,h+2\alpha,\ldots,h+(n-1)\alpha\}$, 
    then $\alpha\not\in H$.
\end{lem}
\begin{proof}
    The case $\alpha<0$ is obvious since $H\subset \mathbb{Z}_{\geq0}$.
    Suppose that $\alpha>0$ and $\alpha\in H$.
    Then $(h+\alpha)+\alpha\in H$ since $\alpha>0$ and $h+\alpha\in\PF(H)$.
    This contradicts $h+2\alpha\in H$.
    Hence, $\alpha\not\in H$.
\end{proof}
\begin{lem}\label{lem:2.ciai,diai}
    Let $s,t\in H$.
    If $s\in \Ap(H)$ and $t\leq_H s$,
    then $t\in\Ap(H)$.
\end{lem}
\begin{proof}
    Suppose contrary that $t\not\in\Ap(H)$.
    Then, there exists $b\in H$ such that $b\equiv t \mod a_0$ and $b<t$.
    Thus, we have $b+(s-t)<t+(s-t)=s$.
    Since $b+(s-t)\in H$ and $b+(s-t)\equiv s\mod a_0$, we obtain $s\not\in\Ap(H)$, which is a contradiction.
\end{proof}

\begin{lem}\label{lem:2.gcd}
    Suppose that $\PF(H)=\{h+\alpha,h+2\alpha,\ldots,h+(n-1)\alpha\}$.
    If $a_0/2+1\leq n$, then $\gcd(a_0,\alpha)=1$.
\end{lem}
\begin{proof}
    Let $l=\gcd(a_0,\alpha)$.
    Suppose contrary that $l>1$.
    Then there exists $k>1$ such that $a_0=kl$.
    From $kl/2=a_0/2<n$, it follows that the integer $k$ is at most $n-1$.
    Let $p_k=h+k\alpha$ for $k=1,2,\ldots,n-1$.
    If $k<n-1$, then it follows that $p_{k+1}=p_1+k\alpha\equiv p_1$.
    This contradicts the fact that, for $l=0,1,\ldots,a_0-1$, $\Ap(H)$ contains exactly one element $\sigma$ such that $\sigma\equiv l\mod a_0$.
\end{proof}
\begin{dfn}
    Suppose that $\gcd(a_0,\alpha)=1$.
    For any $k\in\mathbb{Z}$, there exists $0\leq i \leq a_0-1$ such that $k\equiv i\alpha\mod a_0$.
    Then we define the map $v_\alpha^{a_0}:\mathbb{Z}\to\{0,1,\ldots,a_0-1\}$ by $v_\alpha^{a_0}(k)=i$.
\end{dfn}
For any $a,b\in\mathbb{Z}$, if $v_\alpha^{a_0}(a)+v_\alpha^{a_0}(b)\leq a_0-1$, then $v_\alpha^{a_0}(a+b)=v_\alpha^{a_0}(a)+v_\alpha^{a_0}(b)$,
and if $v_\alpha^{a_0}(a)-v_\alpha^{a_0}(b)\geq0$, then $v_\alpha^{a_0}(a)-v_\alpha^{a_0}(b)=v_\alpha^{a_0}(a-b)$.
By a simple observation, we see that $v_\alpha^{a_0}(a)+v_{-\alpha}^{a_0}(a)=a_0$ for any integer $a\not\in a_0\mathbb{Z}$.

Next, we introduce some lemmas which will be used to determine the  defining ideal $I_H$ of $k[H]$.
\begin{lem}\label{prop:2.JinIH}
    $($\cite[Proposition 2.10]{Kien2025}$)$
    Let $J$ be a graded ideal of $S=k[X_0,X_1,\ldots,X_{n-1}]$ such that $J\subseteq I_H$.
    If $\dim S/(J+(X_1))=a_0$, then $I_H=J$.
\end{lem}
\begin{lem}\label{lem:2.I2M<IH}
    $($\cite[Corollary 2.12]{Kien2025}$)$
    Let $f_0,f_1,\ldots,f_{n-1},g_0,g_1,\ldots,g_{n-1}\in S$ 
    and let
    \begin{equation*}
        M=
        \begin{pmatrix}
            f_0&f_1&\cdots&f_{n-1}\\
            g_0&g_1&\cdots&g_{n-1}
        \end{pmatrix}.
    \end{equation*}
    If 
    \begin{equation*}
        \deg \varphi(g_0)-\deg \varphi(f_0)=\cdots = \deg \varphi(g_{n-1})-\deg \varphi(f_{n-1}),
    \end{equation*}
    then $\mathrm{I}_2(M)\subseteq I_H$.
\end{lem}

%% file: 3.ProofA.tex
In the following, let $H$ be a numerical semigroup with embedding dimension $n\geq3$ and multiplicity $a_0$.
Assume that there exist $h>0$ and $\alpha\neq0$ such that $\PF(H)=\{h+\alpha,h+2\alpha,\ldots,h+(n-1)\alpha\}$ and $\gcd(a_0,\alpha)=1$.
We denote  $v_\alpha^{a_0}$ by $v_\alpha$.
Let
\begin{equation*}
    G=\{a_0,a_1,a_2,\ldots,a_i,a_1',a_2',\ldots,a_j',a_1'',a_2'',\ldots,a_k''\}
\end{equation*}
be the minimal system of generators of $H$, 
where
\begin{itemize}
    \item $G\cap\PFa(H)=\{a_1',a_2',\ldots,a_j'\}$.
    \item $v_\alpha(a_0)=0<v_\alpha(a_1)<\cdots<v_\alpha(a_i)<v_\alpha(a_1')<\cdots<v_\alpha(a_j')<v_\alpha(a_1'')<\cdots<v_\alpha(a_k'')$.
\end{itemize}
Let $\varphi:S=k[X_0,X_1,\ldots,X_i,X_1',\ldots,X_j',X_1'',\ldots,X_k'']\to k[H]$ be 
the ring homomorphism defined by
$\varphi(X_s)=t^{a_s},\varphi(X_t')=t^{a_t'},\varphi(X_u'')=t^{a_u''}$
and $I_H=\ker \varphi$.

\begin{thm}\label{thm:3.claimA}
    Under the above assumption, if
    \begin{equation*}
        v_\alpha(a_i)+v_{-\alpha}(a_1'')<n-1,
    \end{equation*}
    then there exist positive integers $c_0,c_1,\ldots,c_i,c_0'',c_1'',\ldots,c_k''$ such that
    \begin{equation*}
        I_H=\mathrm{I}_2(M_+~M_0~M_-)
    \end{equation*}
    where
    \begin{equation*}
        M_+=
        \begin{pmatrix}
            X_0^{c_0} & X_1^{c_2} & \cdots &X_{i-1}^{c_{i-1}} & X_i^{c_i}\\
            X_1 & X_2 & \cdots & X_i & X_1'
        \end{pmatrix}
    \end{equation*}
    \begin{equation*}
        M_0=
        \begin{pmatrix}
            X_1'&X_2'&\cdots&X_{j-1}'\\
            X_2'&X_3'& \cdots &X_j' 
        \end{pmatrix}
    \end{equation*}
    \begin{equation*}
        M_-=
        \begin{pmatrix}
            X_j'&X_1''&X_2''&\cdots&X_{k-1}''&X_k''\\
            X_1''^{c_1''}&X_2''^{c_2''}&X_3''^{c_3''}&\cdots&X_k''^{c_k''}&X_0^{c_0''}
        \end{pmatrix}.
    \end{equation*}
\end{thm}
\begin{rem}
    If $i=0$, we regard the assumption $v_\alpha(a_i)+v_{-\alpha}(a_1'')<n-1$ as $v_{-\alpha}(a_1'')<n-1$ and
    
    \begin{equation*}
        M_+=\begin{pmatrix}
           X_0^{c_0}\\
           X_1'
        \end{pmatrix}.
    \end{equation*}
    Analogously, if $k=0$, we regard the assumption $v_\alpha(a_i)+v_{-\alpha}(a_1'')<n-1$ as $v_{\alpha}(a_i)<n-1$ and
    
    \begin{equation*}
        M_-=
        \begin{pmatrix}
            X_j'\\
            X_0^{c_0''}
        \end{pmatrix}.
    \end{equation*}
    We will see in Lemma \ref{lem:3.l<r} that $j\neq0$.
\end{rem}
This Theorem \ref{thm:3.claimA} proves the main Theorem \ref{thm:1.mainthm}.
\begin{proof}[Proof of Theorem \ref{thm:1.mainthm}]
    It suffices to prove that the assumption of Theorem \ref{thm:3.claimA} is satisfied when $a_0/2+1\leq n$,
    that is, $\gcd(a_0,\alpha)=1$ and $v_\alpha(a_l)+v_{-\alpha}(a_m)<n-1$ for any $a_l,a_m$ with $v_\alpha(a_l)<\min v_\alpha(\PFa(H))$ and $v_{-\alpha}(a_m)<\min v_{-\alpha}(\PFa(H))$.
    $\gcd(a_0,\alpha)=1$ follows from Lemma \ref{lem:2.gcd}.
    We show that $v_\alpha(a_l)+v_{-\alpha}(a_m)<n-1$.
    Since $\#\PFa(H)=n-1$ and $a_0/2+1\leq n$, we have $a_0/2\leq n-1=\#\PFa(H)$.
    Then, it follows that 
    \begin{equation*}
        \#(\Ap(H)\setminus(\PFa(H)\cup\{0\}))=a_0-(n-1)-1
        < n-1.
    \end{equation*}
    Then, by the choice of $a_l$ and $a_m$, we conclude
    \begin{align*}
        v_\alpha(a_l)+v_{-\alpha}(a_m)&\leq (\min v_\alpha(\PFa(H)))-1)+(\min v_{-\alpha}(\PFa(H)))-1)\\
        &\leq \#\Ap(H)\setminus(\PFa(H)\cup\{0\})<n-1,
    \end{align*}
    as desired.
\end{proof}

We now prove Theorem \ref{thm:3.claimA}.
Before the proof, we note one remark and introduce some notation.
\begin{rem}\label{rem:3.switching}
    We may switch the role of $a_1,a_2,\ldots,a_i$ and $a_1'',a_2'',\ldots,a_k''$ by changing $\alpha$ with $-\alpha$.
    In fact, by setting  $\beta=-\alpha$ and $h'=h+n\alpha$,
    we have 
    \begin{equation*}
        \PF(H)=\{h'+\beta,h'+2\beta,\ldots,h'+(n-1)\beta\}
    \end{equation*}
    and
    \begin{equation*}
        v_\beta(a_0)=0<v_\beta(a_k'')<v_\beta(a_{k-1}'')<\cdots<v_\beta(a_1'')<v_\beta(a_j')<\cdots<v_\beta(a_1')<v_\beta(a_i)<\cdots<v_\beta(a_1).
    \end{equation*}
 \end{rem}
Let $\sigma_k:=\min \{a\in H:v_\alpha(a)=k\}(\in\Ap(H))$.
Then, we see that 
\begin{equation*}
    \Ap(H)=\{\sigma_0,\sigma_1,\ldots,\sigma_{a_0-1}\}
\end{equation*}
and there is $1\leq s\leq a_0-1$ such that
\begin{equation*}
    \PFa(H)=\{\sigma_s,\sigma_{s+1},\ldots,\sigma_{s+n-1}\}.
\end{equation*}
\begin{exam}\label{exam:runexam1}
    Let $H=\langle a_0=13, a_1=67,a_2=69,a_3=71, a_1'=144,a_2'=146,a_1''=37\rangle$.
    Then, we have 
    \begin{equation*}
        \PFa(H)=\{138,140,142,144,146,148\}
    \end{equation*}
    with $h=123$ and $\alpha=2$,
    and
    \begin{equation*}
        \Ap(H)=\left\{
        \begin{split}
\sigma_0=0,\sigma_1=67,\sigma_2=69,\sigma_3=71,\sigma_4=138,\sigma_5=140,\sigma_6=142,\\
        \sigma_7=144,\sigma_8=146,\sigma_9=148,\sigma_{10}=111,\sigma_{11}=74,\sigma_{12}=37
    \end{split}
    \right\}.
    \end{equation*}
\end{exam}
We first study the structure of the elements of  the Ap\'ery set of $H$. 
\begin{lem}\label{lem:3.a+a'notinAp}
    Let $\mathbf{0}\neq(c_1,c_2,\ldots,c_i)\in\mathbb{Z}_{\geq0}^i$ and $\mathbf{0}\neq(c_1'',c_2'',\ldots,c_k'')\in\mathbb{Z}_{\geq0}^k$.
    Then, $\sum_{l=1}^ic_la_l+\sum_{l=1}^kc_l''a_l''\not\in\Ap(H)$.
\end{lem}
\begin{proof}
    By Lemma \ref{lem:2.ciai,diai}, it is enough to show $a_l+a_m''\not\in\Ap(H)$ for any $1\leq l\leq i$ and $1\leq m\leq k$.
    If $a_l+a_m''\in\Ap(H)$, then there exists $\eta\in\Ap(H)$ such that $b=a_l+a_m''+\eta\in\PFa(H)$.
    Put $b^+=b-a_l$ and $b^-=b-a_m''$.
    Both elements belong to $\Ap(H)\setminus\PFa(H)$.
    Since $v_\alpha(a_l)<v_\alpha(b)$ and $b-a_l\not\in\PFa(H)$,
    we have $0<v_\alpha(b-a_l)<\min v_\alpha(\PFa(H))$.
    By Remark \ref{rem:3.switching}, we have $0<v_\beta(b-a_m'')<\min v_\beta(\PFa(H))$.
    It holds that $v_\beta(b^+)>\max v_\beta(\PFa(H))$ and $v_\beta(b^-)<\min v_\beta(\PFa(H))$.
    Then,
    \begin{align*}
        \#\PF(H)&=n-1\\
        &>v_\alpha(a_l)+v_\beta(a_m'')\\
        &=v_\alpha(a_l)+(a_0-v_\alpha(a_m''))\\
        &=a_0-v_\alpha(a_m''-a_l)\\
        &=a_0-v_\alpha(b^+-b^-)\\
        &=v_\beta(b^+-b^-)
        >\#\PF(H),
    \end{align*}
    a contradiction, where the first inequality follows from the assumption of Theorem \ref{thm:3.claimA}.
\end{proof}
Let $r=\min v_\alpha(\PFa(H))+v_\alpha(a_i)$ and $r'=\min v_\beta(\PFa(H))+v_\beta(a_1'')$.
Note that
\begin{equation*}
    \min v_\alpha(\PFa(H))<r\leq a_0-r'<\max v_\alpha(\PFa(H)).
\end{equation*}
\begin{lem}\label{lem:3.l<r}
    The following holds.
    \def\labelenumi{(\theenumi)}
    \begin{enumerate}
        \item If $0\neq\sigma_l\in\langle a_1,\ldots,a_i\rangle$, then $l<r$.
        \item If $0\neq\sigma_l\in\langle a_1'',\ldots,a_k''\rangle$, then $l>a_0-r'$.
        \item If $r\leq l\leq a_0-r'$, then $\sigma_l\in\{a_1',\ldots,a_j'\}$.
    \end{enumerate}
\end{lem}
\begin{proof}
    (1) Let $\sigma_l=a_{t_1}+a_{t_2}+\cdots+a_{t_q}$.
    We prove it by induction on $q$.
    When $q=1$, we have $\sigma_l=a_{t_1}$.
    Thus, we obtain $l=v_\alpha(a_{t_1})<\min v_\alpha(\PFa(H))<r$.
    Next, assume $q>1$, and let $\tau=a_{t_1}+a_{t_2}+\cdots+a_{t_{q-1}}$.
    By Lemma \ref{lem:2.ciai,diai}, we have $\tau\in\Ap(H)$.
    By the induction hypothesis, we have $v_\alpha(\tau)<r$,
    and since $\tau\not\in\PFa(H)$ it follows that $v_\alpha(\tau)<\min v_\alpha(\PFa(H))$.
    Hence, we obtain $l=v_\alpha(\sigma_l)\leq v_\alpha(\tau)+v_\alpha(a_i)<r$.

    (2) follows from (1) and Remark \ref{rem:3.switching}.

    (3) 
    By considering the contrapositive of (1) and (2),
    if $r\leq l \leq a_0-r'$, then by Lemma \ref{lem:3.a+a'notinAp} we have $\sigma_l\not\in\langle a_1,\ldots,a_i, a_1'',\ldots,a_k''\rangle$.
    Thus, it follows that $\sigma_l\in\langle a_1',\ldots,a_j'\rangle$.
    Since $a_1',\ldots,a_j'\in\PFa(H)$, we obtain $\sigma_l\in\{a_1',\ldots,a_j'\}$.
\end{proof}
\begin{lem}\label{prop:3.X-o,Y-o}
    Suppose that $X=\sum_{l=1}^ic_la_l\in\PFa(H),~v_\alpha(X)<r$ and $\omega=\sum_{l=1}^id_la_l\in\Ap(H)$ with $d_l\leq c_l$ for all $l$. 
    If $Y=X-\alpha\in\PFa(H)$, then $Y-\omega\in\Ap(H)$.
\end{lem}
\begin{proof}
    Since $\omega,X-\omega\in\langle a_1,\ldots,a_i\rangle$, 
    it follows that $v_\alpha(\omega),v_\alpha(X-\omega)<r$.
    Thus, we have $v_\alpha(Y-\omega)<v_\alpha(X-\omega)<r$.
    Suppose $Y-\omega\in H\setminus\Ap(H)$.
    Then, there exists $b\in H$ such that $v_\alpha(b)=v_\alpha(Y-\omega)$ and $b<Y-\omega$.
    It follows that $b+\omega<Y$ and $b+\omega\in H$.
    Since $v_\alpha(b+\omega)=v_\alpha(Y)$,
    this contradicts $Y\in\Ap(H)$.
    Hence $Y-\omega\in\Ap(H)$ or $Y-\omega\not\in H$.
    Therefore, there is an integer $t\geq0$ such that $Y-\omega+ta_0\in\Ap(H)$.
    Since $X-\omega\not\in\PFa(H)$, we have $Y-\omega +ta_0\not\in \PFa(H)$, and it follows that $Y-\omega+ta_0\not\in\{ a_1',\ldots,a_j'\}$.
    Also, since $v_\alpha(Y-\omega+ta_0)<v_\alpha(X-\omega)<r$, 
    we have $Y-\omega+ta_0\not\in\langle a_1'',\ldots,a_k''\rangle$ by Lemma \ref{lem:3.l<r} (2).
    Thus, $Y-\omega+ta_0\in\langle a_1,\ldots,a_i\rangle$.
    By Corollary \ref{cor:2.+eta}, there exists $\eta\in\Ap(H)$ such that $Y-\omega +ta_0+\eta\in\PFa(H)$.
    By Lemma \ref{lem:3.a+a'notinAp}, it follows that $Y-\omega+ta_0+\eta\in\langle a_1,\ldots,a_i\rangle$.
    Therefore by Lemma \ref{lem:3.l<r} we have $v_\alpha(Y-\omega+ta_0+\eta)<r\leq \max v_\alpha(\PFa(H))$,
    and it follows that $X-\omega+ta_0+\eta=(Y-\omega+ta_0+\eta)+\alpha\in\PFa(H)$.
    Since $X-\omega,\eta\in\Ap(H)$, we conclude $t=0$.
\end{proof}
\begin{cor}\label{cor:3.YnotG}
    Suppose that $X\in\langle a_1,\ldots,a_i\rangle\cap\PFa(H)$ and $Y=X-\alpha\in\PFa(H)$.
    Then, $Y\not\in G$.
\end{cor}
\begin{proof}
    By the assumption, there is $(c_1,\ldots,c_i)\in\mathbb{Z}_{\geq0}^i$ such that $(c_1,\ldots,c_i)\neq\mathbf{e}_l(1\leq l\leq i)$ and $X=\sum_{l=1}^ic_la_l$.
    We can take $\textbf{0}\neq(d_1,\ldots,d_i)\in\mathbb{Z}_{\geq0}$ satisfying $(d_1,\ldots,d_i)\leq(c_1,\ldots,c_i)$.
    Let $\omega=\sum_{l=1}^id_la_l$.
    Then, $X-\omega\in\Ap(H)$.
    By Lemma \ref{prop:3.X-o,Y-o}, we obtain $Y-\omega\in \Ap(H)$.
    Since $Y=(Y-\omega)+\omega$, $Y$ can always be expressed as a sum of at least two generators.
\end{proof}
By Remark \ref{rem:3.switching}, we can prove the following statements in the same way as in Lemma \ref{prop:3.X-o,Y-o} and Corollary \ref{cor:3.YnotG}.
\begin{lem}
    Suppose that $X=\sum_{l=1}^kc_l''a_l''\in\PFa(H) ,~v_\beta(X)<r'$ and
    $\omega=\sum_{l=1}^kd_l''a_l''\in\Ap(H)$ with $d_l''\leq c_l''$ for all $l$.
    If $Y=X-\beta\in\PFa(H)$, then $Y-\omega\in\Ap(H)$.
\end{lem}

\begin{cor}\label{cor:3YnotG2}
    Suppose that $X\in\langle a_1'',\ldots,a_k''\rangle\cap\PFa(H)$ and $Y=X-\beta\in\PFa(H)$.
    Then, $Y\not\in G$.
\end{cor}
By Corollaries \ref{cor:3.YnotG} and \ref{cor:3YnotG2}, there are positive integers $g<h$ satisfying
\begin{equation}\label{eq:3.Ap}
   \Ap(H)=\{0,\sigma_1,\ldots,\sigma_g,\sigma_{g+1},\ldots,\sigma_{h-1},\sigma_{h},\ldots\sigma_{a_0-1}\} 
\end{equation}
with $\sigma_1,\ldots,\sigma_{g}\in\langle a_1,\ldots,a_i\rangle,\sigma_{g+1}=a_1',\sigma_{g+2}=a_2',\ldots,\sigma_{h-1}=a_j',$
and $\sigma_{h},\ldots,\sigma_{a_0-1}\in\langle a_1'',\ldots,a_k''\rangle$.
\begin{exam}
    Let $H$ be as in Example \ref{exam:runexam1}.
    Then, we have $\PFa(H)=\{138,140,142,144,146,148\}$ with $h=123$ and $\alpha=2$.
    Thus, the Ap\'ery set of $H$ is 
    \begin{equation*}
        \Ap(H)=\{0,67,69,71,138,140,142,144,146,148,111,74,37\}
    \end{equation*}
    with $67,69,71,138,140,142\in\langle67,69,71\rangle$ and
    $148,111,74,37\in\langle37\rangle$.
\end{exam}

Our next goal is to determine the elements in $\PFa(H)$.
We split $\Ap(H)$ as follows.
\begin{align*}
    \Ap_+(H)&:=\{a\in \Ap(H):0<v_{\alpha}(a)<v_{\alpha}(a_1')\};\\
    \Ap_0(H)&:=\{a\in \Ap(H):v_{\alpha}(a_1')\leq v_{\alpha}(a)\leq v_{\alpha}(a_j')\};\\
    \Ap_-(H)&:=\{a\in \Ap(H):v_{\alpha}(a)>v_{\alpha}(a_j')\}.
\end{align*}
We have $\Ap(H)=\Ap_+(H)\cup \Ap_0(H) \cup \Ap_-(H)$, and $a_s\in \Ap_+(H),a_t'\in \Ap_0(H),a_u''\in \Ap_-(H)$.
\begin{exam}
     Let $H$ be as in Example \ref{exam:runexam1}.
    Then, 
    \begin{align*}
    \Ap_+(H)&=\{67,69,71,138,140,142\};\\
    \Ap_0(H)&=\{144,146\};\\
    \Ap_-(H)&=\{148,111,74,37\}.
    \end{align*}
\end{exam}
\begin{lem}
    Let $t$ be a positive integer.
    \def\labelenumi{(\theenumi)}
    \begin{enumerate}
        \item If $t<r$, then $\sigma_t+\alpha\leq\sigma_{t+1}$.
        \item If $t<a_0-r'$, then $\sigma_t+\beta\leq\sigma_{t-1}$.
    \end{enumerate}
\end{lem}
\begin{proof}
    By Remark \ref{rem:3.switching}, it suffices to prove (1).
    When $\sigma_t\in\PFa(H)$, it is obvious that statement holds since $\sigma_{l+1}\in\PFa(H)$.
    Assume that $\sigma_t\not\in\PFa(H)$.
    Since $\sigma_t\in\Ap(H)$, by Corollary \ref{cor:2.+eta},
    there exists $\eta\in\Ap(H)$ such that $\sigma_t+\eta\in\PFa(H)$.
    Since $t<r$, we have $\sigma_t\in\langle a_1,a_2,\ldots,a_i\rangle$,
    and therefore $\sigma_t+\eta\in\langle a_1,a_2,\ldots,a_i\rangle$ by Lemmas \ref{lem:3.a+a'notinAp} and \ref{lem:3.l<r} (2),
    which proves $v_\alpha(\sigma_t+\eta)<r$.
    Then, we have $\sigma_t+\eta+\alpha\in\PFa(H)$.
    Since $v_\alpha(\sigma_t+\eta+\alpha)=v_\alpha(\sigma_{t+1}+\eta)$ and $\sigma_{t+1}+\eta\in H$, it follows that $\sigma_{t}+\eta+\alpha\leq \sigma_{t+1}+\eta$.
    Hence, $\sigma_t+\alpha\leq \sigma_{t+1}$.
\end{proof}

\begin{lem}\label{lem:3.ainH+,a''inH-}
    The following holds.
    \def\labelenumi{(\theenumi)}
    \begin{enumerate}
        \item If $\sum_{l=1}^ic_la_l\in\Ap(H)\setminus
        \{0\}$ with $c_l\geq0$, then $\sum_{l=1}^ic_la_l\in \Ap_+(H)$.
        \item If $\sum_{l=1}^kc_l''a_l''\in \Ap(H)\setminus\{0\}$ with $c_l''\geq0$, then $\sum_{l=1}^kc_l''a_l''\in \Ap_-(H)$.
    \end{enumerate}
\end{lem}
\begin{proof}
    By Remark \ref{rem:3.switching}, it suffices to prove (1).
    By Lemma \ref{lem:3.l<r}(1), there is an integer $m<r$ such that $\sigma_m=\sum_{l=1}^ic_la_l$.
    By Lemma \ref{lem:3.l<r}(2), Corollary \ref{cor:3.YnotG} and equation (\ref{eq:3.Ap}), it follows that $\sigma_m\in \Ap_+(H)$.
\end{proof}
\begin{lem}\label{lem:3.sigma<a}
    The following holds.
    \def\labelenumi{(\theenumi)}
    \begin{enumerate}
        \item If the element $\sigma_l\in \Ap_+(H)$, then
        $\sigma_l\not\geq_H a$ for any $a\in G$ with $v_\alpha(a)>v_\alpha(\sigma_l)$.
        \item If the element $\sigma_l\in \Ap_-(H)$, then
        $\sigma_l\not\geq_H a$ for any $a\in G$ with $v_\beta(a)>v_\beta(\sigma_l)$.
    \end{enumerate}
\end{lem}
\begin{proof}
    By Remark \ref{rem:3.switching}, it suffices to prove (1).
    Suppose contrary that there is a generator $a\in G$ such that $\sigma_l\geq_H a$ and $v_\alpha(a)>v_\alpha(\sigma_l)$.  
    Since $\sigma_l,a\in\Ap(H)$ and $\sigma_l-a\in H$ by the assumption, we have $\sigma_l-a\in \Ap(H)$.
    If $a\in \Ap_-(H)$, then, we have $\sigma_l\in\langle a_1'',a_2'',\ldots,a_k''\rangle$ by Lemma \ref{lem:3.a+a'notinAp} since $\sigma_l=a+(\sigma_l-a)$.
    However, by Lemma \ref{lem:3.ainH+,a''inH-}, we have $\sigma_l\in \Ap_-(H)$.
    This is contradiction.
    Thus, we may assume $a\in \{a_1,\ldots,a_i\}$ such that $v_\alpha(\sigma_l)<v_\alpha(a)$.
    Since $\sigma_l=(\sigma_l-a)+a$ and $a,\sigma_l-a\in\Ap(H)$, we have $b:=\sigma_l-a\in \langle a_1,a_2,\ldots,a_i\rangle$ by Lemma \ref{lem:3.a+a'notinAp}. 
    Then, by Lemma \ref{lem:3.ainH+,a''inH-} we have $b\in \Ap_+(H)$, but this contradicts 
    $v_\alpha(\sigma_l)=v_\alpha((\sigma_l-a)+a)>v_\alpha(a)$.
\end{proof}

Note that by this lemma, we have $v_\alpha(a_1)=1$ and $v_\alpha(a''_k)=a_0-1$.
\begin{lem}\label{lem:3.pfexpression}
    Let $\PFa(H)=\{p_1,p_2,\ldots,p_{n-1}\}$ with $p_l=p_1+(l-1)\alpha$.
    \def\labelenumi{(\theenumi)}
    \begin{enumerate}
        \item There exist integers $c_1,\ldots,c_i$ such that $p_m=a_m+\sum_{l=m}^i(c_l-1)a_l$ for any $m=1,2,\ldots,i$.
        \item There exist integers $c_1'',\ldots,c_k''$ such that $p_{n-1-k+m}=a_m''+\sum_{l=1}^m(c_l''-1)a_l''$ for any $m=1,2,\ldots,k$.
    \end{enumerate}
\end{lem}
\begin{proof}
    By Remark \ref{rem:3.switching}, it suffices to prove (1).
    To begin with, we show a following claim.
    \begin{clm*}
        If $p_l\in\langle a_1,\ldots,a_i\rangle$, 
        then $\min \{s:a_s\leq_Hp_l\}>\min\{s:a_s\leq_Hp_{l-1}\}$.
    \end{clm*}
    We assume $p_l=a_s+\sum_{t=s}^id_ta_t$.
    Let $\omega=\sum_{t=s}^id_ta_t$.
    We have $p_l-\omega=a_s$.
    By Lemma \ref{prop:3.X-o,Y-o}, $p_l-\omega\in\Ap(H)$ implies that $p_{l-1}-\omega\in\Ap(H)$.
    By Lemma \ref{lem:3.sigma<a}, we have $\omega'=p_{l-1}-\omega\in\langle a_1,\ldots,a_{s-1}\rangle$.
    Since $p_{l-1}=\omega'+\omega$, it follows that $\min \{s:a_s\leq_Hp_l\}>\min\{s:a_s\leq_Hp_{l-1}\}$.
    
    Next, we show that $\PFa(H)$ consists of $i$ elements of $\Ap_+(H)$, $j$ elements of $\Ap_0(H)$, and $k$ elements of $\Ap_-(H)$.
    Suppose contrary that at least $i+1$ elements of $\Ap_+(H)$ are contained in $\PFa(H)$.
    Assume that $p_1,p_2,\ldots,p_m\in\PFa(H)$ and $m\geq i+1$.
    By Claim, we obtain
    \begin{equation*}
        \min \{s:a_s\leq_Hp_1\}<\min \{s:a_s\leq_Hp_2\}<\cdots
        <\min \{s:a_s\leq_Hp_m\}.
    \end{equation*}
    However, this contradicts $\#\{a_1,a_2.\ldots,a_i\}=i$.
    Thus, there are at most $i$ elements contained in $\PFa(H)\cap\Ap_+(H)$.
    Analogously, by Remark \ref{rem:3.switching}, there are at most $k$ elements contained in $\PFa(H)\cap\Ap_-(H)$.
    By the assumption of Theorem \ref{thm:3.claimA}, we have $j$ elements $a_1',\ldots,a_j'\in\PFa(H)$ and $\#\PFa(H)=n-1=i+j+k$.
    Hence, there are exactly $i$ elements in $\Ap_+(H)$,
    $j$ elements in $\Ap_0(H)$,
    and $k$ elements in $\Ap_-(H)$.
    By this fact and equation (\ref{eq:3.Ap}),
    it follows that $p_1,p_2,\ldots,p_i\in\Ap_+(H),p_{i+1}=a_1',p_{i+2}=a_2',\ldots,p_{i+j}=a_j'$
    and $p_{i+j+1},p_{i+j+2},\ldots,p_{i+j+k}\in\Ap_-(H)$.

    We now prove the statement by induction on $m$.
    We consider the case $m=i$.
    By the Claim, we have
    \begin{equation*}
        \min \{s:a_s\leq_Hp_1\}<\min \{s:a_s\leq_Hp_2\}<\cdots
        <\min \{s:a_s\leq_Hp_i\}.
    \end{equation*}
    If $\min \{s:a_s\leq_Hp_i\}<i$, then we cannot take $p_1$ satisfying this inequality.
    Thus, we have $p_i=c_ia_i$ with $c_i>0$.
    Assume that there are $c_{m+1},c_{m+2},\ldots,c_i$ such that $p_s=a_s+\sum_{l=s}^i(c_l-1)a_l$ for $s=m+1,m+2,\ldots,i$.
    We prove that there is $c_m\geq0$ such that $p_m=a_m+\sum_{l=m}^i(c_l-1)a_l$.
    Let $\omega=\sum_{l=m+1}^i(c_l-1)a_l\in\Ap(H)$.
    By Claim, we have
    \begin{equation*}
        \min \{s:a_s\leq_Hp_1\}<\cdots
        <\min \{s:a_s\leq_Hp_m\}<\min \{s:a_s\leq_Hp_{m+1}\}=m+1.
    \end{equation*}
    Thus, we have $\min \{s:a_s\leq_Hp_m\}=m$.
    Since $\min \{s:a_s\leq_Hp_m\}=m$ and $v_\alpha(p_m-\omega)<v_\alpha(p_{m+1}-\omega)=v_\alpha(a_{m+1})$,
    we conclude $p_m=c_ma_m+\omega=a_m+\sum_{l=m}^i(c_l-1)a_l$ with $c_m>0$.
\end{proof}
By this lemma, all elements of $\PFa(H)$ are as follows
\begin{equation*}
    \PFa(H)=\left\{
    \begin{split}
    &a_1+\sum_{l=1}^i(c_l-1)a_l,a_2+\sum_{l=2}^i(c_l-1)a_l,\ldots,c_ia_i,\\
    &a_1',\ldots,a_j',\\
    &c_1''a_1'',a_2''+\sum_{l=1}^2(c_l''-1)a_l'',\ldots,a_k''+\sum_{l=1}^k(c_l''-1)a_l
    \end{split}
    \right\}.
\end{equation*}
\begin{exam}
    Let $H$ be as in Example \ref{exam:runexam1}.
    Then,
    \begin{equation*}
        \PFa(H)=\{138,140,142,144,146,148\}
        =\{a_1+a_3,a_2+a_3,2a_3,a_1',a_2',4a_1''\}.
    \end{equation*}
    
\end{exam}

We next give the lower bound for $\#\Ap(H)$.
\begin{dfn}
    By Lemmas \ref{prop:3.X-o,Y-o} and \ref{lem:3.pfexpression}, the element
    \begin{equation}\label{eq:3.pfexpress}
        a_m+\sum_{l=m}^id_la_l
    \end{equation}
    with $0\leq d_l\leq c_l-1$ for $l=m,m+1,\ldots,i$ is contained in $\Ap(H)$.
    Let $B_+$ be a set of all elements of the form (\ref{eq:3.pfexpress}).
    We call an expression (\ref{eq:3.pfexpress}) the \textbf{canonical expression} of an element of $B_+$.
    We define $B_-$ using $a_1'',a_2'',\ldots a_k''$ in the same way as $B_+$.
\end{dfn}
\begin{lem}\label{lem:3.canonicalexpress}
    The following statements holds.
    \def\labelenumi{(\theenumi)}
    \begin{enumerate}
        \item If $a,a'\in B_+$ have different canonical expressions, then $a\neq a'$.
        \item If $a,a'\in B_-$ have different canonical expressions, then $a\neq a'$.
    \end{enumerate}
\end{lem}
\begin{proof}
    By Remark \ref{rem:3.switching}, it suffices to prove (1).
    Suppose contrary that $a=a'$.
    We assume that $a,a'\in B_+$ have different canonical expression.
    If a generator $a_m$ appears in canonical expressions of both $a$ and $a'$, $a-a_m$ and $a'-a_m$ satisfy the same condition.
    We may assume that $a=\sum_{l=1}^{u}d_{s_l}a_{s_l}$ and $a'=\sum_{l=1}^{v}d_{t_l}a_{t_l}$ 
    with $\{s_1,s_2,\ldots,s_u\}\cap\{t_1,t_2,\ldots,t_v\}=\varnothing$.
    Let $M:=\max(\{s_1,s_2,\ldots,s_u\}\cup\{t_1,t_2,\ldots,t_v\})$ and without loss of generality, 
    we may assume $a\geq_Ha_M$.
    Let $p=a_M+\sum_{l=M}^i(d_l-1)a_l$ and let $q\in \PFa(H)$ such that $q\geq_Ha'$.
    Then, we obtain 
    \begin{align*}
        q-a'+a&=\left(q-a'+a+a_M+\sum_{l=M}^i(d_l-1)a_l\right)-\left(a_M+\sum_{l=M}^i(d_l-1)a_l\right)\\
        &=\left\{\left(q-\sum_{l=M}^i(d_l-1)a_l\right)-a'+(a-a_M)\right\}+a_M+\sum_{l=M}^i(d_l-1)a_l\\
        &=\left\{\left(q-\sum_{l=M}^i(d_l-1)a_l\right)-a'+(a-a_M)\right\}+p.
    \end{align*}
    We have $p\leq_H q$, but this is in conflict  with $p\in\PFa(H)$.
\end{proof}
\begin{cor}\label{cor:3.Apnumber}
    $\#\Ap(H)\geq \left(\sum_{l=1}^i\prod_{m=l}^ic_m\right)+j+\left(\sum_{l=1}^k\prod_{m=1}^lc_m''\right)+1$.
\end{cor}
\begin{proof}
     Counting the number of possible canonical expressions,
     it follows from Lemmas \ref{lem:3.pfexpression} and \ref{lem:3.canonicalexpress} that $\#B_+=\sum_{l=1}^i\prod_{m=l}^ic_m$ and $\#B_-=\sum_{l=1}^k\prod_{m=1}^{l}c_m$ respectively.
     Since $\Ap(H)\supseteq B_+\cup\{a_1',a_2',\ldots,a_j'\}\cup B_-\cup\{0\}$,
     it follows that $\#\Ap(H)\geq\#B_++\#\{a_1',a_2',\ldots,a_j'\}+\#B_-+1$.
\end{proof}
Finally, we show that $\mathrm{I}_2(M)=I_H$.
\begin{proof}[Proof of Theorem \ref{thm:3.claimA}]
    Since $h+\alpha+a_0=a_1+\sum_{l=1}^i(c_l-1)a_l$ and $v_\alpha(a_1)=v_\alpha(\alpha)=1$, we have $v_\alpha(h)=v_\alpha(\sum_{l=1}^i(c_l-1)a_l)$.
    Since $\alpha\not\in H$ by Lemma \ref{lem:2.alphanotinH},
    we have $\alpha<a_1$.
    Thus, we obtain $\sigma_{v_\alpha(h)}=\sum_{l=1}^i(c_l-1)a_l<h+a_0$.
    In other words, there exists an integer $c_0\in\mathbb{Z}_{>0}$ such that $h+a_0=\sum_{l=1}^i(c_l-1)a_l+c_0a_0$.
    Analogously, there exists an integer $c_0''\in\mathbb{Z}_{>0}$ such that $h+n\alpha+a_0=\sum_{l=1}^k(c_l''-1)a_l''+c_0''a_0$.
    Let
    \begin{gather*}
         M_+'=
        \begin{pmatrix}
            X_0^{c_0}\prod_{l=1}^iX_l^{c_l-1}&X_1^{c_1}\prod_{l=2}^iX_l^{c_l-1}&\cdots&X_{i-1}^{c_{i-1}}X_i^{c_i-1}&X_i^{c_i}\\
            X_1^{c_1}\prod_{l=2}^iX_l^{c_l-1}&X_2^{c_2}\prod_{l=3}^iX_l^{c_l-1}&\cdots&X_i^{c_i}&X_1'
        \end{pmatrix}\\
         M_0=
        \begin{pmatrix}
            X_1'&X_2'&\cdots &X_{j-1}'\\
            X_2'&X_3'&\cdots &X_j'
        \end{pmatrix}\\
        M_-'=
        \begin{pmatrix}
            X_j'&X_1''^{c_1''}&X_1''^{c_1''-1}X_2''^{c_2''}&\cdots&X_{k-1}''^{c_{k-1}''}\prod_{l=1}^{k-2}X_l''^{c_l''-1}&X_{k}''^{c_{k}''}\prod_{l=1}^{k-1}X_l''^{c_l-1''}\\
            X_1''^{c_1''}&X_1''^{c_1''-1}X_2''^{c_2''}&X_3''^{c_3''}\prod_{l=1}^2X_l''^{c_l''-1}&\cdots&X_{k}''^{c_{k}''}\prod_{l=1}^{k-1}X_l''^{c_l-1''}&X_0^{c_0''}\prod_{l=1}^{k}X_l''^{c_l-1''}
        \end{pmatrix}.
    \end{gather*}
    By the construction of $M_+',M_0,M_-'$, in each column of $M_+,M_0,M_-$, if we denote the top entry by $f$ and the bottom entry by $g$,
    then $\deg \varphi(f)-\deg \varphi(g)=\alpha$ is constant.
    Since $M:=(M_+~M_0~M_-)$ is obtained by dividing each column of $(M_+'~M_0'~M_-')$ by the greatest common devisor of its two entries,
    by Lemma\ref{lem:2.I2M<IH},
    we have $\mathrm{I}_2(M)\subseteq I_H$.
    
    We now prove $\mathrm{I}_2(M)=I_H$ using Lemma \ref{prop:2.JinIH}.
    Since $\mathrm{I}_2(M)\subseteq I_H$,
    we have $\dim_k(S/(\mathrm{I}_2(M)+(X_0)))\geq\dim_k(S/(I_H+(X_0)))$.
    Since $I_H=\ker \varphi$, it follows that $\varphi(I_H+(X_0))=(t^{a_0})$.
    In other words, $S/(I_H+(X_0))\cong k[H]/(t^{a_0})$.
    Therefore, we obtain
    \begin{equation*}
        \dim_k(S/(I_H+(X_0)))=\dim_k(k[H]/(t^{a_0}))=a_0.
    \end{equation*}
    Thus, we have 
    \begin{equation}\label{ineq:3.dimgeqa0}
    \dim_k(S/(\mathrm{I}_2(M)+(X_0)))\geq a_0.    
    \end{equation}
    Let $LT(\mathrm{I}_2(M)+(X_0))$ be the initial ideal of $\mathrm{I}_2(M)+(X_0)$
    with respect to the weighted degree reverse lexicographic order with
    $X_0>X_1>\cdots>X_i>X_1'>\cdots>X_j'>X_1''>\cdots>X_k''$.
    Recall $\dim_k(S/(\mathrm{I}_2(M)+(X_0)))=\dim_k(S/LT(\mathrm{I}_2(M)+(X_0)))$
    \cite[Proposition 4 of Chapter 5 \S 3]{Cox2025}.
   
    This initial ideal $LT(\mathrm{I}_2(M)+(X_0))$ contains the following monomials.
    \begin{gather*}
    X_0,X_sX_{s'}^{c_{s'}},X_sX_t',X_sX_u'',
    X'_tX_{t'}',X_t'X_u'',X_u''^{c_{u}''}X_{u'}''\\
        (1\leq s \leq s'\leq i,1\leq t\leq t'\leq j,1\leq u\leq u'\leq k).
    \end{gather*}
    Then, $S/(I_H+(X_0))$ is generated, as a vector space, by monomials that are not divisible by the above monomials.
    They are
    \begin{gather*}
        X_1^{s_1}\prod_{l=2}^iX_l^{s_l}(1\leq s_1\leq c_1,0\leq s_l \leq c_l-1(2\leq l \leq i))\\
        X_2^{s_2}\prod_{l=3}^iX_l^{s_l}(1\leq s_2\leq c_2,0\leq s_l \leq c_l-1(3\leq l \leq i))\\
        \vdots\\
        X_{i-1}^{s_{i-1}}X_i^{s_i}(1\leq s_{i-1}\leq c_{i-1},0\leq s_i \leq c_i-1)
        \end{gather*}
        \begin{gather*}
        X_i^{s_i}(1\leq s_i\leq c_i),\\
        X_1',X_2',\ldots,X_j',\\
        X_1''^{t_1}(1\leq t_1\leq c_1'')\\
        X_{2}''^{t_{2}}X_1''^{t_1}(1\leq t_{2}\leq c_{2}'',0\leq t_1 \leq c_1''-1)\\
        \vdots\\
        X_k''^{t_k}\prod_{l=1}^{k-1}X_l''^{t_l}(1\leq t_k\leq c_k'',0\leq t_l \leq c_l''(1\leq l \leq k-1) \\
        1.
    \end{gather*}
    Thus, we have
    \begin{equation}
        \dim_k(S/\mathrm{I}_2(M)+(X_0))\leq \left(\sum_{l=1}^i\prod_{m=l}^ic_m\right)+j+\left(\sum_{l=1}^k\prod_{m=1}^lc_m\right)+1.
    \end{equation}
    By Corollary \ref{cor:3.Apnumber}, it follows that $\#\Ap(H)\geq \left(\sum_{l=1}^i\prod_{m=l}^ic_m\right)+j+\left(\sum_{l=1}^k\prod_{m=1}^lc_m\right)+1$.
    Consequently, we have
    \begin{equation}\label{ineq:3.dimandAp}
        \dim_k(S/(\mathrm{I}_2(M)+(X_0)))\leq\#\Ap(H)=a_0.
    \end{equation}
    By the inequalities (\ref{ineq:3.dimgeqa0}) and (\ref{ineq:3.dimandAp}), we obtain $\dim_k(S/(\mathrm{I}_2(M)+(X_0)))=a_0.$
    By Lemma \ref{prop:2.JinIH}, we conclude $\mathrm{I}_2(M)=I_H$.
\end{proof}

\begin{exam}
    Let $H$ be as in Example \ref{exam:runexam1}.
    We have $\PFa(H)=\{a_1+a_3,a_2+a_3,2a_3,a_1',a_2',4a_1''\}$,
     $h+a_0=a_3+5a_0$, and $h+n\alpha+a_0=3a_1''+3a_0$.
     Then, we have
     \begin{equation*}
         M'=\begin{pmatrix}
             X_0^5X_3&X_1X_3&X_2X_3&X_3^2&X_1'&X_2'&X_1''^4\\
             X_1X_3&X_2X_3&X_3^2&X_1'&X_2'&X_1''^4&X_0^3X_1''^3
         \end{pmatrix},
     \end{equation*}
     and
     \begin{equation*}
         M=\begin{pmatrix}
             X_0^5&X_1&X_2&X_3^2&X_1'&X_2'&X_1''\\
             X_1&X_2&X_3&X_1'&X_2'&X_1''^4&X_0^3
         \end{pmatrix}.
     \end{equation*}
     By the proof of Theorem \ref{thm:3.claimA}, it follows that $I_H=\mathrm{I}_2(M)$.
\end{exam}

%% file: 4.proofB.tex
In this section, using the same technique as in the Theorem \ref{thm:3.claimA}, we prove another result that supports Conjecture \ref{conj}.
Let $H$ be a numerical semigroup with embedding dimension $n$ and multiplicity $a_0$.
Assume that 
\def\labelenumi{(\theenumi)}
\begin{enumerate}
    \item $\PF(H)=\{h+\alpha,h+2\alpha,\ldots,h+(n-1)\alpha\}$ with $h>0$, $\alpha>0$, and $\gcd(a_0,\alpha)=1$.
    \item The element $h+(n-1)\alpha+a_0$ is contained in the minimal generating system of $H$ and $v_\alpha(h+(n-1)\alpha)=a_0-1$.
\end{enumerate}
Let $G$ be the minimal system of generators of $H$.
Set $G=\{a_0,a_1,\ldots,a_i,a_1',\ldots,a_j'\}$, where
\begin{itemize}
    \item $v_\alpha(a_0)=0<v_\alpha(a_1)<\cdots<v_\alpha(a_i)<v_\alpha(a_1')<\cdots<v_\alpha(a_j')=a_0-1$
    \item $G\cap \PFa(H)=\{a_1',\ldots,a_j'\}$.
\end{itemize}
Let $\varphi:S=k[X_0,X_1,\ldots,X_i,X_1',\ldots,X_j']\to k[H]$ be $\varphi(X_s)=t^{a_s}$ and $\varphi(X_t')=t^{a_t'}$, and $I_H=\ker\varphi$.
\begin{thm}\label{thm:3.2claim2}
    Under the above assumption, 
    there exist integers $c_0,c_1,\ldots,c_{n-1},c_0'$ satisfying the following equation.
    \begin{equation*}
        I_H=\mathrm{I}_2
        \begin{pmatrix}
            X_0^{c_0} & X_1^{c_2} & \cdots &X_{i-1}^{c_{i-1}} & X_i^{c_i}&X_1'&\cdots&X_{j-1}'&X_j'\\
            X_1 & X_2 & \cdots & X_i &X_1'&X_2'&\cdots&X_j'& X_0^{c_0'}
        \end{pmatrix}
    \end{equation*}
\end{thm}
The proof proceeds in the same way as in the proof of Theorem \ref{thm:3.claimA}, so we only sketch the proof.
By the assumption, we have 
\begin{equation*}
    v_\alpha(\PFa(H))=\{
a_0-(n-1),a_0-(n-2),\ldots,a_0-1\}.
\end{equation*}
Let $p_i=h+i\alpha\in\PF(H)$.
\begin{lem}\label{prop:3.2X-o,Y-o}
    Suppose that $X\in\PFa(H)\setminus\{a_1',\ldots,a_j'\},~\omega\in\Ap(H)$.
    If $Y=X-\alpha\in\PFa(H)$ and $X-\omega\in\Ap(H)$, then $Y-\omega\in\Ap(H)$.
\end{lem}
\begin{proof}
    Assume that $Y-\omega\in H\setminus\Ap(H)$.
    Then, there is $b\in H$ such that $v_\alpha(b)=v_\alpha(Y-\omega)$ and $b<Y-\omega$.
    Thus, since $b+\omega<Y$ and $b+\omega\in H$,
    we have $Y\not\in\Ap(H)$.
    This is a contradiction.

    We may assume that $Y-\omega\in\Ap(H)$ or $Y-\omega\not\in H$.
    Thus, there is an integer $t\geq0$ such that $Y-\omega+ta_0\in\Ap(H)$.
    By Corollary \ref{cor:2.+eta}, we have $(Y-\omega+ta_0)+\eta\in\PFa(H)$ for some $\eta\in\Ap(H)$.
    Since $p_{n-1}+a_0=a_j'$ and $G$ is the minimal system of generators,
    it follows that $(Y-\omega+ta_0)+\eta\neq a_j'$.
    Hence, we have
    \begin{equation*}
        (Y-\omega+ta_0)+\eta+\alpha=X-\omega+ta_0+\eta\in\PFa(H).
    \end{equation*}
    $X-\omega+\eta\in H$ implies that $t=0$.
    Thus, we obtain $Y-\omega\in\Ap(H)$.
\end{proof}
\begin{lem}\label{lem:3.2sigmak+a<sigmak+1}
    For $1\leq k <a_0-1$, we have $\sigma_k+\alpha\leq \sigma_{k+1}$.
\end{lem}
\begin{proof}
    If $\sigma_k\in\PFa(H)$, then this is obvious.
    Assme that $\sigma_k\not\in\PFa(H)$.
    By Corollary \ref{cor:2.+eta}, there exists $\eta\in\Ap(H)$ such that $\sigma_k+\eta\in\PFa(H)$.
    Since $p_{n-1}+a_0=a_{n-1}$ and $G$ is the minimal system of generators,
    we have $\sigma_k+\eta\neq p_{n-1}$.
    Thus, it follows that $\sigma_k+\eta +\alpha\in\PFa(H)$.
    Since $v_\alpha(\sigma_k+\alpha)=v_\alpha(\sigma_{k+1})$, 
    we have $\sigma_k+\eta+\alpha\leq \sigma_{k+1}+\eta$.
    Therefore, we conclude that $\sigma_k+\alpha\leq \sigma_{k+1}$.
\end{proof}
\begin{cor}\label{lem:3.2sigmaincreasing}
    If the element $\sigma_k\in\Ap(H)$, then $\sigma_l\not\geq_H a$ for any $a\in G$ with $v_\alpha(a)>v_\alpha(\sigma_l)$.
\end{cor}
\begin{proof}
    The statement is trivial by Lemma \ref{lem:3.2sigmak+a<sigmak+1} since we assume $\alpha>0$.
\end{proof}
Note that $v_\alpha(a_1)=1$.
\begin{lem}\label{lem;3.2pfexpress2}
    \def\labelenumi{(\theenumi)}
    \begin{enumerate}
        \item There exist integers $c_1,c_2,\ldots,c_i\geq1$ such that
        \begin{equation*}
            \PFa(H)=\left\{a_1+\sum_{l=1}^i(c_l-1)a_l,a_2+\sum_{l=2}^i(c_l-1)a_l,\ldots,c_ia_i,a_1',a_2',\ldots,a_j'\right\}.
        \end{equation*}
        \item $\#\Ap(H)\geq(\sum_{l=1}^i\prod_{m=l}^ic_m)+j+1$.
    \end{enumerate}
\end{lem}
The proof of Lemma \ref{lem;3.2pfexpress2} is carried out in the same way as in the proof of Lemma \ref{lem:3.pfexpression} and Corollary \ref{cor:3.Apnumber}.
\begin{proof}[Proof of Theorem \ref{thm:3.2claim2}]
    Since $h+\alpha+a_0=a_1+\sum_{l=1}^i(c_l-1)a_l$ and $v_\alpha(a_1)=v_\alpha(\alpha)=1$,
    we have $v_\alpha(h)=v_\alpha(\sum_{l=1}^i(c_l-1)a_l)$.
    Since $\alpha\not \in H$ by Lemma \ref{lem:2.alphanotinH}, we have $\alpha<a_1$.
    Thus, we obtain $\sigma_{v_\alpha(h)}=\sum_{l=1}^i(c_l-1)a_l<h+a_0$.
    In other words, there exists an integer $c_0$ such that $h=\sum_{l=1}^i(c_l-1)a_l+c_0a_0$ with $c_0>0$.
    By the assumption of Theorem \ref{thm:3.2claim2}, $v_\alpha(a_j')=a_0-1$.
    Thus, we have $h+n\alpha+a_0=a_j'+\alpha+a_0\equiv0\mod a_0$.
    There exists an integer $c_0'>0$ such that $a_j'+\alpha+a_0=c_0'a_0$.

   Let
    \begin{equation*}
        M_+'=
        \begin{pmatrix}
            X_0^{c_0}\prod_{l=1}^iX_l^{c_l-1}&X_1^{c_1}\prod_{l=2}^iX_l^{c_l-1}&\cdots&X_{i-1}^{c_{i-1}}X_i^{c_i-1}&X_i^{c_i}\\
            X_1^{c_1}\prod_{l=2}^iX_l^{c_l-1}&X_2^{c_2}\prod_{l=3}^iX_l^{c_l-1}&\cdots&X_i^{c_i}&X_1'
        \end{pmatrix}
    \end{equation*}
    \begin{equation*}
        M_0=
        \begin{pmatrix}
            X_1'&X_2'&\cdots &X_{j-1}'&X_j'\\
            X_2'&X_3'&\cdots &X_j'&X_0^{c_0'}
        \end{pmatrix}
    \end{equation*}
    By construction of $M_+',M_0$, in each column if we denote the top entry by $f$ and the bottom entry by $g$,
    then $\deg\varphi(f)-\deg\varphi(g)=\alpha$ is constant.
    Since $M:=(M_+~M_0)$ is obtained by dividing each column of $(M_+~M_0)$ by the greatest common devisor of its two entries,
    by Lemma \ref{lem:2.I2M<IH}, we have $\mathrm{I}_2(M)\subseteq I_H$.
    
    We now prove $\mathrm{I}_2(M)=I_H$ using Lemma \ref{prop:2.JinIH}.
    By comparing the dimension of $S/(\mathrm{I}_2(M)+(X_0))$ and $a_0=\#\Ap(H)$ using Lemma \ref{lem;3.2pfexpress2},
    we can prove $\dim_k(S/(\mathrm{I}_2(M)+(X_0)))=a_0$ in the same way as in the proof of Theorem \ref{thm:3.claimA}.
    By Lemma \ref{prop:2.JinIH}, we conclude $\mathrm{I}_2(M)=I_H$.
\end{proof}

%% file: 5.Example.tex
We end the paper with a few examples.
\begin{exam}
    Let $H=\langle a_0=13,a_1=14,a_2=15,a_1'=46,a_2'=47,a_1''=24,a_2''=25\rangle$.
    Then
    $\PF(H)=\{31,32,33,34,35,36\}$,
    so it satisfies the assumption of Theorem \ref{thm:3.claimA}
    with $h=30$ and $\alpha=1$.
    Thus, we have
    \begin{equation*}
        I_H=\mathrm{I}_2\left(\begin{matrix}
    X_0&X_1&X_2^3&X_1'&X_2'&X_1''&X_2''\\
    X_1&X_2&X_1'&X_2'&X_1''^2&X_2''&X_0^2
    \end{matrix}\right).
    \end{equation*}
\end{exam}
\begin{exam}
    Let $H=\langle a_0=9,a_1=19,a_2=48,a_1'=106,a_2'=116\rangle$.
    Then
    $\PF(H)=\{77,87,97,107\}$.
    so it satisfies the assumption of Theorem \ref{thm:3.2claim2}
    with $h=67$ and $\alpha=10$.
    Thus, we have
    \begin{equation*}
        I_H=\mathrm{I}_2\left(\begin{matrix}
    X_0&X_1^2&X_2^2&X_1'&X_2'\\
    X_1&X_2&X_1'&X_2'&X_0^{14}
    \end{matrix}\right).
    \end{equation*}
\end{exam}